\documentclass[a4paper,12pt]{amsart}
\usepackage[utf8]{inputenc}
\usepackage[T1]{fontenc}
\usepackage[UKenglish]{babel}
\usepackage[margin=20mm]{geometry}
\usepackage{color}
\usepackage{bbm,bm}

\usepackage{subfigure}

\usepackage{graphicx}

\usepackage{amsmath,amssymb,amsfonts,amsthm}
\usepackage{mathrsfs,eucal,dsfont}
\usepackage{verbatim,enumitem}

\usepackage{hyperref,url}
\usepackage{amssymb, amsmath, amsthm, amscd}

\usepackage{eucal,mathrsfs,dsfont}
\usepackage{color}

\renewcommand{\leq}{\le}
\renewcommand{\geq}{\ge}

\newcommand{\e}{\text{e}}

\newcommand{\I}{\mathds 1}

\def\d{{\rm d}}
\def\<{\langle}
\def\>{\rangle}

\newtheorem{theorem}{Theorem}[section]
\newtheorem{lemma}[theorem]{Lemma}
\newtheorem{proposition}[theorem]{Proposition}
\newtheorem{corollary}[theorem]{Corollary}
\numberwithin{equation}{section}

\theoremstyle{definition}
\newtheorem{definition}[theorem]{Definition}

\newtheorem{remark}[theorem]{Remark}

\begin{document}
\allowdisplaybreaks
\title[Spectral dimensions for critical long-range percolation]
{\bfseries  Spectral dimensions for one-dimensional critical long-range percolation}

\author{Zherui Fan  \qquad  Lu-Jing Huang}

\thanks{\emph{Z. Fan:}
School of Mathematical Sciences, Peking University, Beijing, China.
  \texttt{1900010670@pku.edu.cn}}

\thanks{\emph{L.-J. Huang:}
School of Mathematics and Statistics \& Key Laboratory of Analytical Mathematics and Applications (Ministry of Education), Fujian Normal University, Fuzhou, China.
  \texttt{huanglj@fjnu.edu.cn}}


\date{}
\maketitle

\begin{abstract}

Consider the critical long-range percolation on $\mathds{Z}$, where an edge connects $i$ and $j$ independently with probability $1-\exp\{-\beta\int_i^{i+1}\int_j^{j+1}|u-v|^{-2}\d u\d v\}$ for $|i-j|>1$ for some fixed $\beta>0$ and with probability 1 for $|i-j|=1$.
We prove that both the quenched and annealed spectral dimensions of the associated simple random walk are  $2/(1+\delta)$, where $\delta\in (0,1)$ is the exponent of the effective resistance in the LRP model, as derived in  \cite[Theorem 1.1]{DFH25+}. Our work addresses an open question from \cite[Section 5]{CCK22}.


\medskip


\end{abstract}
\allowdisplaybreaks


\section{Introduction}

The study of random walks on percolation clusters in the integer lattice $\mathds{Z}^d$ has attracted significant interest over the years (see e.g.\ \cite{Bar04,MFGW89,SS04}).
In particular, random walks on long-range percolation (LRP) models have emerged as a notable area of research due to their characteristic ``long-range jumps'' and the complex phase transition phenomena they exhibit.

In this paper, we will concentrate on random walks in the one-dimensional critical long-range percolation (LRP) on $\mathds{Z}$.
Specifically, we consider LRP on $\mathds{Z}$, where edges $\langle i,j\rangle $ with $|i-j|=1$ (i.e.\ $i$ and $j$ are nearest neighbors) are present, while edges $\langle i,j\rangle $ with $|i-j|>1$ (which we refer to as long edges in what follows) occur independently with probability
\begin{equation}\label{def-LRP}
1-\exp\left\{-\beta\int_i^{i+1}\int_j^{j+1}\frac{1}{|u-v|^2}\d u\d v\right\}.
\end{equation}
Here, $\beta> 0$ is a parameter of this LRP model, which we refer to as the $\beta$-LRP model, and
denote by $\mathcal{E}$ the random edge set.
For convenience, we will also use $\sim$ to denote edges,
that is, $i\sim j$ implies $\langle i,j\rangle \in \mathcal{E}$.
Throughout the paper, we suppose this model is built on a probability space with probability measure $\mathds{P}$ and expectation $\mathds{E}$.

Next, let $X=(X_n)_{n\geq 0}$ be the discrete-time simple random walk on the $\beta$-LRP model.
For each $x\in \mathds{Z}$, we denote by $\mathbf{P}_x$ and $\mathbf{E}_x$ the law of $X$ starting from $x$ and the corresponding expectation, respectively.
The transition probabilities of $X$ are given by
\begin{equation*}\label{def-rw}
  \mathbf{P}[X_{n+1}=y\mid X_n=x]=\frac{1}{\text{deg}(x)}\quad \text{for all $n\geq 0$ and }y\sim x,
\end{equation*}
where $\text{deg}(x):=\#\{y\in \mathds{Z}:\ y\sim x\}$ is the degree of $x$ in the model.
It is worth emphasizing that $\mathds{P}$-a.s. we have $\text{deg}(x)<\infty$ for all $x\in \mathds{Z}$ from \eqref{def-LRP}. Hence, the associated random walk is well defined. We then define the transition density (also known as the discrete-time heat kernel) of $X$ by
\begin{equation*}
p_n(x,y)=\frac{1}{\text{deg}(y)}\mathbf{P}_x[X_n=y]\quad \text{for all $n\geq 1$ and all } x,y\in \mathds{Z}.
\end{equation*}
Additionally, we define the associated quenched spectral dimension as
\begin{equation*}
  d_s^{(q)}(\beta)=-2\lim_{n\rightarrow \infty}\frac{\log p_{2n}(0,0)}{\log n}
\end{equation*}
and the annealed spectral dimension as
\begin{equation*}
  d_s^{(a)}(\beta)=-2\lim_{n\to\infty}\frac{\log\mathds{E}[p_{2n} (0,0)]}{\log n},
\end{equation*}
if the limits exist.

Our main result establishes the estimates for the heat kernels and spectral dimensions of the $\beta$-LRP model, as detailed in the following theorem. Let $\delta \in (0,1)$ be the exponent of the effective resistance in the $\beta$-LRP model defined in \cite[Theorem 1.1]{DFH25+}, see Lemma \ref{Lam-order} below for more details.

\begin{theorem}\label{thm-mr}
For all $\beta>0$, the $\beta$-LRP satisfies the following properties.
\begin{itemize}
\item[\rm(1)] There exist constants $0<c_1,c_2,\gamma_1<\infty$ {\rm(}depending only on $\beta${\rm)} such that $\mathds{P}$-a.s., for all $n\in \mathds{N}$ large enough,
    \begin{equation*}
      c_1n^{-\frac{1}{1+\delta}}(\log n)^{-\gamma_1}\leq p_{2n}(x,x)\leq c_2n^{-\frac{1}{1+\delta}}(\log n)^{\gamma_1}.
    \end{equation*}

\item[\rm(2)]There exist constants $0<c_3,c_4<\infty$  {\rm(}depending only on $\beta${\rm)}  such that for all $n\in \mathds{N}$,
    \begin{equation*}
      c_3n^{-\frac{1}{1+\delta}}\leq \mathds{E}\left[p_{2n}(x,x)\right]\leq c_4n^{-\frac{1}{1+\delta}}.
    \end{equation*}
\end{itemize}
In particular, we have that $\mathds{P}$-a.s. $d_s^{(q)}(\beta)=2/(1+\delta)$, and $d_s^{(a)}(\beta)=2/(1+\delta)$.
\end{theorem}

\begin{remark}
(1) As we mentioned in the abstract, our main result establishes the diagonal decay of the heat kernel and the spectral dimensions for random walk in $\beta$-LRP model, thereby addressing Open Question $5$ from \cite[Section 5]{CCK22}.

  (2) It has become standard that many other random walk estimates can be proved using the same techniques in \cite{KM08}.
  For example, after $n$ steps, the random walk $X=(X_k)_{k\geq 0}$  will have a Euclidean distance from the origin that, with high probability, is $n^{1/(1+\delta)+o(1)}$.
Furthermore, the total number of distinct vertices visited by the random walk after $n$ steps is also $n^{1/(1+\delta)+o(1)}$. We omit the details.
\end{remark}

The main technique for proving Theorem \ref{thm-mr} is based on ideas from \cite{KM08}, specifically by establishing suitable upper and lower bound estimates for the volume and effective resistance in the model to derive properties of the associated random walk. Therefore, we will introduce some notations and present the key results from \cite{KM08} in Section \ref{sect-pre} as preparation. We will then present the proof of Theorem \ref{thm-mr} in Section \ref{sect-proof}.

\subsection{Related work}

The study of random walks on percolation clusters within $\mathds{Z}^d$ has become a significant area of interest in mathematical physics and probability.
Initial studies predominantly focused on the nearest-neighbor percolation model.
For instance, \cite{GKZ93} proved that the random walk on the infinite cluster is recurrent in the two-dimensional case and
transient when $d\geq 3$. Afterwards, \cite{MFGW89} first established the diffusion scaling limits under the annealed law for the random walk on the infinite cluster of the two-dimensional percolation.
Subsequent works, such as \cite{SS04} (for $d\geq 4$) and \cite{BB07,MP07} (for $d\geq 2$) have proven invariance principles  governing random walks on supercritical percolation clusters.
Additionally, \cite{MR04} and \cite{Bar04} provided quenched Gaussian heat kernel bounds for the infinite cluster.
When random walks exhibit strong recurrence, a comprehensive theoretical framework has been developed to derive diagonal heat kernel estimates.
These techniques have been crucial in understanding the spectral dimension, which characterizes the decay rate of the diagonal heat kernel, see e.g. \cite{BJKS08,KN09,Ku14}.

It is important to note that all the studies mentioned so far focus on nearest-neighbour percolation, which only considers edges with unit Euclidean length in $\mathds{Z}^d$.
A natural extension is to allow for the existence of longer-range edges, a model known as long-range percolation, which was introduced by \cite{S83,ZPL83}.
Specifically, consider a sequence $\{p_x\}_{x\in \mathds{Z}^d}$, where $p_x\in [0,1]$ and $p_x=p_{x'}$ for all $x,x'\in \mathds{Z}^d$ such that $|x|=|x'|$. Assume that
\begin{equation*}
  0<\beta:=\lim_{|x|\rightarrow \infty}p_x|x|^s<\infty
\end{equation*}
for some $s>0$. Then the LRP on $\mathds{Z}^d$ is defined by edges $\langle x,y\rangle$ occurring independently with probability $p_{x-y}$.

We now review progress on behavior for random walks on LRP models.
In \cite{B02}, it was shown that for $s\geq 2d$ with $d\in \{1,2\}$, the random walk is recurrent; in contrast, for $d<s<2d$ with $d\in \{1,2\}$, the walk is transient, as shown by a renormalization argument.
\cite{KM08} employed methods to estimate volumes and effective resistances from points to boxes, yielding the corresponding heat kernel estimates for the case where $d=1,\ s>2$ and $p_1=1$.
In \cite{CS12},  non-Gaussian upper bounds for the heat kernel were derived for the scenarios where  $s\in (d,d+2)$ for $d\geq 2$ and $s\in (1,2)$ for $d=1$.
Based on this, the authors further established that when $s\in (d,d+1)$, the random walk converges to an $\alpha$-stable process with $\alpha=s-d$.
Notably, in the case $d=1$ and $s>2$, \cite{CS12} established that the random walk converges to a Brownian motion.
These results were later extended in \cite{BT24} to the case of $d\geq 2$ and $s\in [d+1,d+2)$. For the case $d>1$ and $s>2d$, quenched invariance principle is also constructed  in \cite{BCKW21}.
More recently, \cite{CCK22} achieved both quenched and annealed bounds for the diagonal heat kernel, allowing them to determine the spectral dimension of the random walk, with the  exception of the case where $(d,s)=(1,2)$.
They also highlighted that the spectral dimension in the case of $d=1$ and $s=2$ exhibits discontinuity.
Specifically, the authors provided quenched and annealed upper bounds for the diagonal heat kernel in this scenario, although these estimates are insufficient to ascertain its spectral dimension.
Our contribution in this work establishes the existence of the spectral dimension for the one-dimensional $\beta$-LRP.

\bigskip

{\bf Notational conventions.}
We denote $\mathds{N}=\{1,2,3,\cdots\}$. For any $i,j\in \mathds{N}\cup \{0\}$ with $i<j$, we will define $[i,j]=\{i,i+1,\cdots, j\}$ and $[i,j)=\{i,i+1,\cdots, j-1\}$. When we refer to an interval $I$, it always mean $I\cap \mathds{Z}$.

For any graph $G=(V,E)$, any vertex set $I\in V$ and any edge set $H\subset E$, we denote $\#I$ and $\#H$ as the numbers of vertices and edges in $I$ and $H$, respectively.

\section{General setting and preparations}\label{sect-pre}

In this section, we primarily introduce some general notations and results, which mainly draw from \cite{KM08}. These will play a crucial role in the proof of Theorem \ref{thm-mr}.

Let $G=(V,E)$ be a random infinite undirected connected graph with  vertex set $V$ and edge set $E$, defined on a probability space $(\Omega, \mathcal{F}, \mathds{P})$.
We assume that $0\in V$, and that $G$ is locally finite, meaning that $\text{deg}(x):=\#\{y\in V:\ \langle y,x\rangle \in E\}<\infty$ for all $x\in V$.
For convenience, we will also use $\sim$ to denote edges, specifically, $x\sim y$ implies $\langle x,y\rangle\in E$.

Let $d(\cdot,\cdot)$ denote a metric on $G$.
Note that $d$ is not necessarily a graph distance; any metric on $G$ may be employed in this section.
For any $x\in V$ and $r\geq 0$, denote
\begin{equation}\label{def-ball}
B_r(x)=\{y:\ d(x,y)<r\}
\end{equation}
as the ball centered at $x$ with radius $r$ with respect to the metric $d$. Define the volume of the ball $B_r(x)$ as
\begin{equation*}
\mathds{V}_r(x)=\sum_{y\in B_r(x)}\text{deg}(y).
\end{equation*}

Next, let $X=(X_n)_{n\geq 0}$ be the discrete-time simple random walk on $G$.
We denote by $\mathbf{P}_x$ and $\mathbf{E}_x$ the law of $X$  starting from $x$ and the corresponding expectation, respectively. Thus, the transition probabilities of $X$ are given by
\begin{equation*}
\mathbf{P}_x[X_1=y]=\frac{1}{\text{deg}(x)} \quad \text{for all } y\sim x.
\end{equation*}
We then define the transition density (also known as the discrete-time heat kernel) of $X$ by
\begin{equation*}
p_n(x,y)=\frac{1}{\text{deg}(y)}\mathbf{P}_x[X_n=y]\quad \text{for all $n\geq 1$ and }x,y\in V.
\end{equation*}
It can be verified that $p_n(x,y)=p_n(y,x)$.

For a subset $A\subset V$, we define the first hitting time as
\begin{equation*}
T_A=\inf\{n\geq 0:\ X_n\in A\},
\end{equation*}
and the first exit time as $\tau_A=T_{A^c}$. For simplicity, we write
\begin{equation*}
\tau_r=\tau_{B_r(0)}=\min\{n\geq 0:\ X_n\notin B_r(0)\}.
\end{equation*}
We also define a quadratic form $(\mathcal{D}, \mathds{H})$ as follows:
\begin{align*}
\mathds{H}&=\left\{f:\ V\rightarrow \mathds{R}\ \big|\ \frac{1}{2}\sum_{x\sim y}(f(x)-f(y))^2<\infty\right\} \\
\mathcal{D}(f,g)&= \frac{1}{2}\sum_{x\sim y}(f(x)-f(y))(g(x)-g(y))\quad \text{for }f,g\in \mathds{H}.
\end{align*}
For any two disjoint subsets $A,B\subset V$, we define the effective resistance between $A$ and $B$ as
\begin{equation*}\label{def-er}
1/R(A,B)=\inf\left\{\mathcal{D}(f,f):\ f\in \mathds{H},\ f|_A=1\ \text{and }f|_B=0 \right\}.
\end{equation*}
In particular, when $A=\{x\}$ is a singleton, we write $R(x, B)$ for $R(\{x\},B)$.
It is worth emphasizing that the effective resistance $R(\cdot,\cdot)$ defines a metric on the graph $G$, see \cite[Section 2.3]{Ki01} for more details.

Let $\phi,\varphi:\mathds{N}\rightarrow [0,\infty)$ be strictly increasing functions with $\phi(1)=\varphi(1)=1$ which satisfy
\begin{equation}\label{cond-phi}
C_1^{-1}\left(\frac{r}{r'}\right)^{d_1}\leq \frac{\phi(r)}{\phi(r')}\leq C_1\left(\frac{r}{r'}\right)^{d_2},\quad
C_2^{-1}\left(\frac{r}{r'}\right)^{\alpha_1}\leq \frac{\varphi(r)}{\varphi(r')}\leq C_2\left(\frac{r}{r'}\right)^{\alpha_2}
\end{equation}
for all $0<r'\leq r<\infty$, where $C_1,C_2\geq 1$, $1\leq d_1\leq d_2$ and $0<\alpha_1\leq \alpha_2\leq 1$. For convenience, we set $\phi(0)=\varphi(0)=0$, $\phi(\infty)=\varphi(\infty)=\infty$, extending them to $\phi,\varphi: [0,\infty]\rightarrow [0,\infty]$ such that $\phi$ and $\varphi$ are continuous, strictly increasing functions. Let $\Psi$ denote the inverse function of $(\phi\varphi)(\cdot)$.

It is well-known that the volume and the effective resistance are key ingredients in studying the properties of the simple random walk on the graph $G$. To facilitate this analysis, we introduce the following definition of a random set of radius of balls that exhibit ``good'' volume and effective resistance estimates.

\begin{definition}\label{def-Jlambda}
For $\lambda>1$, define
\begin{equation*}
\begin{aligned}
  J(\lambda)&=\left\{r\in [1,\infty]:\ \lambda^{-1}\phi(r)\leq \mathds{V}_r(0)\leq \lambda \phi(r),\ R(0,B_r(0)^c)\geq \lambda^{-1}\varphi(r), \right.\\
  &\quad \quad \quad \quad \quad \quad \quad R(0,y)\leq \lambda \varphi(d(0,y))\text{ for all }y\in B_r(0)\big\}.
  \end{aligned}
  \end{equation*}
\end{definition}

We now give the following {\bf Assumption A} concerning the graph $G$.
\begin{itemize}
\item[(A1)] There exist $\lambda_0>1$ and $p(\lambda)$ which goes to 0 as $\lambda\rightarrow \infty$ such that
\begin{equation*}
\mathds{P}[r\in J(\lambda)]\geq 1-p(\lambda)\quad \text{for each }r\geq 1,\ \lambda \geq \lambda_0.
\end{equation*}

\item[(A2)] There exists a constant $C_{*,1}<\infty$ such that for all $r\geq 1$,
$$
\mathds{E}\left[R(0,B_r(0)^c)\mathds{V}_r(0)\right]\leq C_{*,1}\phi(r)\varphi(r).
$$

\item[(A3)] There exist constants $q_0>0$ and $C_{*,2}<\infty$ such that for all $\lambda>1$,
\begin{equation*}
p(\lambda)\leq \frac{C_{*,2}}{\lambda^{q_0}}.
\end{equation*}
\end{itemize}

The following result is from \cite{KM08}, which provides the estimate for the diagonal transition density of the random walk.

\begin{proposition}[{\cite[Proposition 1.4]{KM08}}]\label{lem-td}
Suppose that {\rm(A1)} and {\rm(A2)} in {\bf Assumption A} hold. Then there exist constants $0<c_1,c_2,c_3, c_4<\infty$ such that the following holds.
\begin{itemize}
\item[\rm(1)] For all $r\geq 1$,
\begin{equation*}
c_1\phi(r)\varphi(r)\leq \mathds{E}\left[\mathbf{E}_0[\tau_r]\right]\leq c_2\phi(r)\varphi(r).
\end{equation*}

\item[\rm(2)] For all $n\geq 1$,
\begin{equation}\label{an-lowb}
\frac{c_3}{\phi(\Psi(n))}\leq \mathds{E}\left[p_{2n}(0,0)\right].
\end{equation}

\item[\rm(3)] For all $n\geq 1$,
\begin{equation*}
  c_4\Psi(n)\leq \mathds{E}\left[\mathbf{E}_0\left[d(0,X_n)\right]\right].
\end{equation*}
\end{itemize}

Assume in addition that there exist $c_5>0$, $\lambda'_0>1$ and $q'_0>2$ such that for all $r\geq 1$ and $\lambda\geq \lambda'_0$,
\begin{equation}\label{cond-VR-uplower}
  \mathds{P}\left[\lambda^{-1}\phi(r)\leq \mathds{V}_r(0),\ R(0,y)\leq \lambda \varphi(d(0,y))\ \text{for all }y\in B_r(0) \right]
  \geq 1-\frac{c_5}{\lambda^{q'_0}}.
\end{equation}
Then there exists a constant $c_6>0$ such that for all $n\geq 1$,
\begin{equation}\label{an-upb}
  \mathds{E}\left[p_{2n}(0,0)\right]\leq \frac{c_6}{\phi(\Psi(n))}.
\end{equation}
\end{proposition}

When we know that $p(\lambda)$ exhibits polynomial decay, as stated in (A3), we can establish certain limit properties for the random walk as follows.

\begin{proposition}[{\cite[Theorem 1.5]{KM08}}]\label{thm-quenched}
Suppose {\rm(A1)} and {\rm(A3)} in {\bf Assumption A} hold.
\begin{itemize}
\item[\rm (1)]  Then there exist constants $0<\gamma_1,\gamma_2,\gamma_3,\gamma_4<\infty$ and a subset $\Omega_0\subset \Omega$ with $\mathds{P}[\Omega_0]=1$ such that the following holds.
    \begin{itemize}
      \item[\rm(a)] For each $\omega\in \Omega_0$ and $x\in V(\omega)$, there exists $N_x(\omega)<\infty$ such that for all $n\geq N_x(\omega)$,
      \begin{equation*}
        \frac{(\log n)^{-\gamma_1}}{\phi(\Psi(n))}\leq p^{\omega}_{2n}(x,x)\leq  \frac{(\log n)^{\gamma_1}}{\phi(\Psi(n))}.
      \end{equation*}
      \item[\rm(b)] For each $\omega\in \Omega_0$ and $x\in V(\omega)$, there exists $r_x(\omega)<\infty$ such that for all $r\geq r_x(\omega)$,
      \begin{equation*}
        (\log r)^{-\gamma_2}\phi(r)\varphi(r)\leq \mathbf{E}_x^\omega[\tau_r]\leq (\log r)^{\gamma_2}\phi(r)\varphi(r).
      \end{equation*}
    \end{itemize}

    \item[\rm(2)] Suppose further that $\phi$ and $\varphi$ satisfy \eqref{cond-phi} and
\begin{equation}\label{cond-phi2}
C_3^{-1}r^D(\log r)^{-m_1}\leq \phi(r)\leq C_3r^D(\log r)^{m_1},\quad C_4^{-1}r^\alpha(\log r)^{-m_2}\leq \varphi(r)\leq C_4 r^{\alpha}(\log r)^{m_2}
\end{equation}
for all $r>0$, where $C_3,C_4\geq 1$, $D\geq 1$, $0<\alpha\leq 1$ and $m_1,m_2>0$. Then the following  holds.
\begin{itemize}
  \item[\rm (a)] The random walk is recurrent and for each $x\in V$, $\mathds{P}$-a.s.\ the quenched spectral dimension
  $$
  d_s^{(q)}(G):=-2\lim_{n\rightarrow \infty}\frac{\log p^\omega_{2n}(x,x)}{\log n}=\frac{2D}{D+\alpha}.
  $$
  \item[\rm (b)] $\lim_{r\rightarrow \infty}\frac{\log \mathbf{E}_x^\omega [\tau_r]}{\log r}=D+\alpha$.
  \item[\rm (c)] Let $W_n=\{X_0,X_1,\cdots, X_n\}$ and let $S_n=\sum_{w\in W_n}{\rm deg}(w)$. Then for each $\omega\in \Omega$ and $x\in V$,
  \begin{equation*}
  \lim_{n\rightarrow \infty}\frac{\log S_n}{\log n}=\frac{D}{D+\alpha}, \quad \mathds{P}_x^\omega\text{-a.s.}
  \end{equation*}
\end{itemize}
\end{itemize}
\end{proposition}

In the final part of this section, we provide some remarks based on \cite[Remark 1.6]{KM08}. These remarks will also be used in the proof of Theorem \ref{thm-mr}.
\begin{remark}\label{remark}
\begin{itemize}

\item[\rm (1)] We introduce a slightly stronger version of $J(\lambda)$ as follows: for $\lambda>1$,
\begin{equation*}
\begin{aligned}
  \widehat{J}(\lambda)&:=\left\{r\in [1,\infty]:\ \lambda^{-1}\phi(r)\leq \mathds{V}_r(0)\leq \lambda \phi(r),\ R(0,B_r(0)^c)\geq \lambda^{-1}\varphi(r), \right.\\
  &\quad \quad \quad \quad \quad \quad \quad R(0,y)\leq \varphi(d(0,y))\text{ for all }y\in B_r(0)\big\}.
  \end{aligned}
  \end{equation*}
Assuming that (A1) in {\bf Assumption A} holds with respect to $\widehat{J}(\lambda)$ and that there exists a constant $C_{*,3}<\infty$ such that
\begin{equation}\label{exp-V}
\mathds{E}\left[1/\mathds{V}_r(0)\right]\leq C_{*,3}/\phi(r)
\end{equation}
(note that this condition is weaker than \eqref{cond-VR-uplower}). Then, both \eqref{an-lowb} and \eqref{an-upb} hold.

\item[\rm (2)] If we choose the resistance metric $R(\cdot,\cdot)$ as the metric $d(\cdot,\cdot)$, then it is clear that $R(0,y)\leq \lambda \varphi(d(0,y))$ holds with $\varphi(x)=x$ and $\lambda=1$.

\end{itemize}
\end{remark}

\section{Proof of Theorem \ref{thm-mr}}\label{sect-proof}

We turn our attention to the $\beta$-LRP model. Our goal is to provide the proof of Theorem \ref{thm-mr} in this section.
To achieve this, we will provide estimates for the diameters of the resistance metric in the $\beta$-LRP model in Section \ref{sect-resist},  which will aid us in applying Propositions \ref{lem-td} and \ref{thm-quenched} to complete the proof of Theorem \ref{thm-mr} in Section \ref{sect-pThm1}.

In the following, we will introduce some notations which will be used repeatedly.
Fix $\beta>0$. For $n\in\mathds{N}$ and $i,j\in [0,n)$, let $R_{[0,n)}(i,j)$ denote the effective resistance $i$ and $j$ restricted to the interval $[0,n)$ in the $\beta$-LRP model. Specifically, we define it as
\begin{equation*}\label{def-R}
1/R_{[0,n)}(i,j)=\inf\left\{\frac{1}{2}\sum_{k,l\in [0,n):\langle k,l \rangle \in \mathcal{E}}(f(k)-f(l))^2:\ f(i)=1\ \text{and }f(j)=0\right\}.
\end{equation*}
Denote
\begin{equation*}\label{def-Lambda}
\Lambda (n,\beta)=\max_{i,j\in [0,n)}\mathds{E}\left[R_{[0,n)}(i,j)\right].
\end{equation*}
For simplicity, we will write $\Lambda(n,\beta)$ as $\Lambda(n)$ when there is no risk of confusion.

\subsection{Diameters of the resistance metric}\label{sect-resist}
The aim of this subsection is to establish the following result, which provides upper bounds for the high moments of the diameters of the resistance metric.

\begin{proposition}\label{max-2r-moment}
  For all $\beta>0$ and $r\in\mathds{N}$, there exists a constant $C_{1,r}<\infty$ {\rm(}depending only on $\beta$ and $r${\rm)} such that
  \begin{equation*}
    \mathds{E}\left[\max_{x\in[0,n)}R_{[0,n)}(0,x)^{r}\right]\leq C_{1,r}\Lambda(n)^{r}.
  \end{equation*}
\end{proposition}

The proof of Proposition \ref{max-2r-moment} will be presented at the end of this subsection. Before that,
we make some preparations. Let us start by introducing the following lemma, which shows that $\Lambda(n)$ exhibits polynomial growth with respect to $n$.

\begin{lemma}[{\cite[Theorem 1.1]{DFH25+}}]\label{Lam-order}
  For all $\beta>0$, there exist constants $\delta=\delta(\beta)>0$ and $0<c_1<C_1<\infty$ {\rm(}all depending only on $\beta${\rm)} such that for all $n\in \mathds{N}$,
  \begin{equation*}
    c_1 n^{\delta(\beta)}\leq  \Lambda(n,\beta)\leq C_1 n^{\delta(\beta)}.
  \end{equation*}
\end{lemma}



In addition, we need the following high moment estimates for the effective resistance. It is important to note that estimates for the second moment have already been provided in \cite[Proposition 5.1]{DFH25+}. Give that the proof shares many similarities with that of \cite[Proposition 5.1]{DFH25+}, we will omit some details in order to concentrate on the differences that arise from the high moments.

\begin{lemma}\label{prop-rthmoment}
  For all $\beta>0$ and $r\in\mathds{N}$, there exists a constant $C_{2,r}<\infty$ {\rm(}depending only on $\beta$ and $r${\rm)} such that for all $N\in \mathds{N}$ and all $i,j\in [0,n)$,
  \begin{equation}\label{2r-moment}
  \mathds{E}\left[R_{[0,n)}(i,j)^{r}\right]\leq C_{2,r}\Lambda(n)^{r}.
  \end{equation}
\end{lemma}

\begin{proof}
We start by introducing some notations that will be used throughout the proof. In fact, these notations are also employed in the proof of \cite[Proposition 5.1]{DFH25+}.

For fixed $m,n\in \mathds{N}$, we say that an interval $[im^n,(i+1)m^n)$ is unbridged if there exists no edge $\langle k, l\rangle$ with $k\in[0,im^n)$ and $l\in [(i+1)m^n,m^{n+1})$. Conversely, if such an edge exists, we say that the interval is bridged by the edge $\langle k,l\rangle$. By the definition, it can be checked that for any $i\in [1,m-2]$,
\begin{equation*}
\mathds{P}\left[ [im^n,(i+1)m^n)\text{ is unbridged}\right]\leq 4i^{-\beta}.
\end{equation*}

We now define a set of edges $\mathcal{B}\subseteq \mathcal{E}$ as follows. For $i,j\in [0,m)$ with $i+1<j$, we add an edge connecting $[im^n,(i+1)m^n)$ to $[jm^n,(j+1)m^n)$ directly into $\mathcal{B}$ if the following conditions are satisfied:
\begin{itemize}

\item[(1)] $[im^n,(i+1)m^n)\sim [jm^n,(j+1)m^n)$, i.e., there exist $k\in [im^n,(i+1)m^n)$ and $l \in [jm^n,(j+1)m^n)$ such that $\langle k,l\rangle \in \mathcal{E}$;

\item[(2)] $[(i-l_1)m^n,(i-l_1+1)m^n)\nsim [(j+l_2)m^n,(j+l_2+1)m^n)$ for all $(l_1,l_2)\in \{0,\cdots, i\}\times \{0,\cdots, m-1-j\}\setminus \{(0,0)\}$.
\end{itemize}
Let $\mathcal{U}'$ be the set of endpoints of edges in $\mathcal{B}$.
We define
\begin{equation*}
\mathcal{U}:=\mathcal{U}'\cup \{0,m^n,\cdots, (m-1)m^n\}\cup\{m^n-1,\cdots, m^{n+1}-1\}.
\end{equation*}
For convenience, we arrange the elements of $\mathcal{U}$ in the ascending order and denote them as $\mathcal{U}=\{x_0,x_1,\cdots, x_u\}$.
By this construction, we have $\#\mathcal{U}\leq 4m$ and $|x_{i-1}-x_i|\leq n-1$.
Now for $x_{i-1},x_i$ with $(x_{i-1},x_i)\neq (km^n-1,km^n)$ for all $k$, we say that the interval $[x_{i-1},x_i]$ is bridged, if there exists an edge $\langle k,l\rangle \in \mathcal{B}$ with $k\leq x_{i-1}<x_i\leq l$.
Note that for a pair $(x_{i-1},x_i)$ that is not of the form $(km^n-1,km^n)$, there must exist a $j\in [0,m-1]$ such that $[x_{i-1},x_i]\subset [jm^n,(j+1)m^n)$. In this case,
$[x_{i-1},x_i]$ and $[jm^n,(j+1)m^n)$ are either both bridged or both not bridged simultaneously.

In the following, we will employ induction to prove that \eqref{2r-moment} holds for $r=2,2^2,2^3,\cdots$.
To this end, by \cite[Proposition 5.1]{DFH25+}, we can see that \eqref{2r-moment} holds for $r=2$. For fixed $k\in \mathds{N}$,
We assume that  \eqref{2r-moment} holds for all $2^l<2^k$ and all $n\in \mathds{N}$.

From now on, we aim to show that \eqref{2r-moment} holds for $r=2^k$.
To this end, we will first estimate $R_{[0,m^{n+1})}(0,s)^{2^k}$ for all $s\in (0,m^{n+1})$.
Denote $u_s=\max\{i\in \{1,\cdots, u\}:\ x_i\leq s\}\leq 4m$ and let
\begin{equation*}
\tau_s=\text{arg}\left(\max_{i\in\{1,\cdots, u_s-1\}}R_{[x_{i-1},x_i]}(x_{i-1},x_i)\vee R_{[x_{u_s},x_{u_s+1}]}(x_{u_s},s)\right).
\end{equation*}
We will say that the interval $[x_{\tau_s-1},x_{\tau_s}]$ is $s$-unbridged if it is not bridged by edges $e=\langle x_{\tau_{s1}}, x_{\tau_{s2}}\rangle\in \mathcal{B}$ with $\tau_{s1}<\tau_{s2}\leq u_s+1$.
Then similar to \cite[(5.13)]{DFH25+}, we can get that for any $s\in (0,m^{n+1})$,
\begin{equation}\label{Rmn}
\begin{aligned}
R_{[0,m^{n+1})}(0,s)^{2^k}
&\leq 2^{2^k}\left((4m)^{2^k}\left(\max_{i\neq \tau_s, i\leq u_s+1}R_{[x_{i-1},x_i]}(x_{i-1},x_i\wedge s)\right)^{2^k}\right.\\
&\quad \quad \quad \quad \quad\left.+\left(\max_{[x_{i-1},x_i]\text{ is $s$-unbridged, }i\leq u_s+1}R_{[0,m^{n+1})}(x_{i-1},x_i\wedge s)\right)^{2^k}\right)\\
&\leq 2^{2^k}\left((4m)^{2^k}\left(\max_{i\neq \tau_s, i\leq u_s+1}R_{[x_{i-1},x_i]}(x_{i-1},x_i\wedge s)\right)^{2^k}\right.\\
&\quad \quad \quad \quad \quad\left.+\left(\max_{[x_{i-1},x_i]\text{ is $s$-unbridged, }i\leq u_s+1}R_{[k_im^n,(k_i+1)m^n)}(x_{i-1},x_i\wedge s)\right)^{2^k}\right),
\end{aligned}
\end{equation}
where in the last inequality we chose $k_i\in \mathds{Z}$ such that $[x_{i-1},x_i]\subset [k_im^n,(k_i+1)m^n)\subset [0,m^{n+1})$ for all $i\leq u_s+1$.

We now bound both terms in the right-hand side of \eqref{Rmn} in expectation.
By using the induction and the argument presented in \cite[(5.14)]{DFH25+}, we have
\begin{align}
&\mathds{E}\left[\left(\max_{i\neq \tau_s, i\leq u_s+1}R_{[x_{i-1},x_i]}(x_{i-1},x_i\wedge s)\right)^{2^k}\right]\nonumber\\
&=\mathds{E}\left[\left(\left(\max_{i\neq \tau_s, i\leq u_s+1}R_{[x_{i-1},x_i]}(x_{i-1},x_i\wedge s)\right)^{2^{k-1}}\right)^2\right]\label{RH-1}\\
&\leq 16m^2\max_{x,y\in [0,m^n)}\left(\mathds{E}\left[\left(R_{[0,m^n)}(x,y)\right)^{2^{k-1}}\right]\right)^2\leq C_{2,2^k}^2m^2\Lambda(m^n)^{2^k}.\nonumber
\end{align}

Moreover, denote
\begin{equation*}
\Gamma(n,r)=\max_{i,j\in [0,m^n)}\mathds{E}\left[R_{[0,m^n)}(i,j)^{r}\right].
\end{equation*}
For the second term in the right-hand side of \eqref{Rmn}, by using the similar argument for the equation above \cite[(5.15)]{DFH25+}, we can get that
\begin{equation}\label{RH-2}
  \begin{aligned}
  &\mathds{E}\left[\left(\max_{[x_{i-1},x_i]\text{ is $s$-unbridged, }i\leq u_s+1}R_{[k_im^n,(k_i+1)m^n)}(x_{i-1},x_i\wedge s)\right)^{2^k}\right]\\
  &=\mathds{E}\left[\mathds{E}\left[\left(\max_{[x_{i-1},x_i]\text{ is $s$-unbridged, }i\leq u_s+1}R_{[k_im^n,(k_i+1)m^n)}(x_{i-1},x_i\wedge s)\right)^{2^k}\mid \mathcal{U}\right]\right]\\
  &\leq \Gamma(n,2^k) \mathds{E}\left[\sum_{i\leq u_s+1}\I_{\{[x_{i-1},x_i]\text{ is $s$-unbridged}\}}\right]\\
  &\leq \Gamma(n,2^k) \mathds{E}\left[\#\{j\in [0,u_s+1): [jm,(j+1)m)\ \text{is $s$-unbridged}\}\right]\\
  &\leq (2+f(\beta,m))\Gamma(n,2^k),
  \end{aligned}
\end{equation}
where
\begin{equation*}
f(\beta,m)=
\begin{cases}
\frac{20}{1-\beta}m^{1-\beta},\quad &\beta<1,\\
10+8\log m,\quad &1\leq \beta\leq 2,\\
20, \quad &\beta>2.
\end{cases}
\end{equation*}

Plugging \eqref{RH-1} and \eqref{RH-2} into \eqref{Rmn}, we derive
\begin{equation}\label{R0s}
\mathds{E}\left[R_{[0,m^{n+1})}(0,s)^{2^k}\right]\leq 2^{2^k}(2+f(\beta,m))\Gamma(n,2^k)+ \widetilde{C}_1(k)m^{2^k+2}\Lambda(m^n)^{2^k},
\end{equation}
where $\widetilde{C}_1(k)<\infty$ is a constant depending only on $\beta$ and $k$.
Note that the derivation of \eqref{R0s} does not depend on endpoints 0 and $s$. Thus, a similar argument yields
that \eqref{R0s} holds for $\mathds{E}\left[R_{[0,m^{n+1})}(r,s)^{2^k}\right]$ for all $0\leq s<r<m^{n+1}$.
As a result,
\begin{equation*}
\Gamma(n+1,2^k)\leq \widetilde{f}(\beta,m)\Gamma(n,2^k)+\widetilde{C}_1(k) m^{2^k+2}\Lambda(m^n)^{2^k},
\end{equation*}
where $\widetilde{f}(\beta,m):=2^{2^k}(2+f(\beta,m))$.
By iterating this inequality over all $l=1,\cdots, n$, we obtain that
\begin{equation}\label{iterate}
\Gamma(n+1,2^k)\leq \widetilde{C}_1(k) m^{2^k+2}\sum_{l=1}^n\left(\widetilde{f}(\beta,m)\right)^{n+1-l}\Lambda(m^l)^{2^k}.
\end{equation}

Now applying a similar argument as for \cite[(5.16)]{DFH25+} to \eqref{iterate}, we can get that there exists a constant $\widetilde{C}_2(k)<\infty$ (depending only on $\beta$ and $k$) such that
\begin{equation}\label{RN}
\mathds{E}\left[R_{[0,N)}(i,j)^{2^k}\right]\leq \widetilde{C}_2(k) \Lambda(N)^{2^k}
\end{equation}
holds for all $i,j\in [0,N)$ and $N=m, m^2,\cdots$. Furthermore, we get that \eqref{RN} holds for general $N\in \mathds{N}$ by the argument in \cite[Page 56]{DFH25+}. That is, \eqref{2r-moment} holds for $r=2^k$.

By the above induction, we obtain that \eqref{2r-moment} holds for all $r=2^k$ with $k\in \mathds{N}$. Finally, for general $r\in \mathds{N}$, let $k\in \mathds{N}$ be the number such that $2^{k-1}<r\leq 2^k$. Then by the H\"older inequality, we have
\begin{equation*}
\max_{i,j\in [0,n)}\mathds{E}\left[R_{[0,n)}(i,j)^{r}\right]\leq \max_{i,j\in [0,n)}\left(\mathds{E}\left[R_{[0,n)}(i,j)^{2^k}\right]\right)^{r/2^k}\leq  \widetilde{C}_3(r)\Lambda(n)^r
\end{equation*}
for all $n\in \mathds{N}$ and some constant $\widetilde{C}_3(r)<\infty$.
\end{proof}

With the above lemmas at hand, we are ready to prove Proposition \ref{max-2r-moment}.

\begin{proof}[Proof of Proposition \ref{max-2r-moment}]
For a fixed $n\in \mathds{N}$, let $m\in\mathds{N}$ be such that $n\in (2^{m-1}, 2^m]$.
For each $x\in [0,n)$, we claim that we can construct a sequence of $\{x_l\}_{l=0}^m$ with $x_0=x$ and $x_m=0$ such that the following properties hold:
  \begin{enumerate}
    \item For each $l\in[0,m]$, $x_l\in 2^{l}\mathds{Z}\cap[0,n)$.
    \medskip

    \item For each $l\in[0,m-1]$, $x_l-x_{l+1}\in \{0, 2^{l}\}$.
  \end{enumerate}
Indeed, to construct the sequence $\{x_l\}_{l=0}^m$, we will proceed by iteration.
Let $x_0=x$. If $x$ is even, set $x_1=x$; otherwise, set $x_1=x-1$.
After defining $x_0,\cdots, x_{l-1}$ for some $l\in[1,m]$, we proceed as follows: If $x_{l-1}$ is divisible by $2^{l}$, set $x_l=x_{l-1}$; otherwise, set $x_l=x_{l-1}-2^{l-1}$.
It can be verified that (1) and (2) hold for each $l\in[1,m]$ and that $x_l\in[0,x]$ for each $l\in[0,m]$. Furthermore, we can see that $x_m=0$ since $x_m\in 2^m\mathds{Z}$ and $0\leq x_m\leq x<n\leq 2^m$.

  Now by the triangle inequality of effective resistance, we can see that for each $x\in [0,n)$,
  \begin{equation}\label{de-Rx}
    \begin{aligned}
      R_{[0,n)}(0,x)\leq \sum_{l=0}^{m-1} R_{[0,n)}(x_{l+1},x_{l})
      &\leq \sum_{l=0}^{m-1} \I_{\{x_{l}\neq x_{l+1}\}} R_{[x_{l+1},x_{l})}(x_{l+1},x_l) \\
      &\leq \sum_{l=0}^{m-1} \max_{y\in[2^{l},2^m)\cap 2^{l}\mathds{Z}}R_{[y-2^{l},y)}(y-2^{l},y).
    \end{aligned}
  \end{equation}

In the following, let $M>0$ be an arbitrary positive constant.
For each $l\in[0,m)$ and each $y\in [2^l,2^m)\cap 2^{l}\mathds{Z}$, define $B_{l,y}$ to be the event that
\begin{equation*}
 R_{[y-2^{l},y)}(y-2^{l},y)>M2^{\delta m}\cdot (2/3)^{(m-l)\delta},
\end{equation*}
where $\delta\in (0,1)$ (depending only on $\beta$) is the constant defined in Lemma \ref{Lam-order}.
Thus, when the event $(\cup_{l=0}^{m-1}\cup_{j=1}^{2^{m-l}-1} B_{l,2^l j})^c$ occurs, we obtain from \eqref{de-Rx} that
\begin{equation*}
  \max_{x\in[0,n)} R_{[0,n)}(0,x)\leq \sum_{l=0}^{m-1}M 2^{\delta m}\cdot (2/3)^{(m-l)\delta}\leq \widetilde{C}_1 M2^{\delta m},
\end{equation*}
where $\widetilde{C}_1:=(1-(2/3)^{\delta})^{-1}<\infty$ depending only on $\beta$. Combining with the translation invariance of the $\beta$-LRP model, we get that
\begin{equation}\label{tail-1}
  \mathds{P}\left[\max_{x\in[0,n)} R_{[0,n)}(0,x)>\widetilde{C}_1 M2^{\delta m}\right]\leq \sum_{l=0}^{m-1}\sum_{j=1}^{2^{m-l}-1}\mathds{P}\left[B_{l,2^{l}j}\right]\leq \sum_{l=0}^{m-1} 2^{m-l}\mathds{P}[B_{l,0}].
\end{equation}

Additionally, it follows from Markov's inequality, Lemmas \ref{prop-rthmoment} and \ref{Lam-order} that for each $l\in [0,m)$ and each  $r\in\mathds{N}$,
\begin{equation*}
  \begin{aligned}
    \mathds{P}[B_{l,0}]&=\mathds{P}\left[R_{[0,2^{l})}(0,2^{l})>M2^{\delta m}\cdot (2/3)^{(m-l)\delta}\right]\\
    &\leq M^{-r} 2^{-\delta mr} (3/2)^{(m-l)\delta r} \mathds{E}\left[R_{[0,2^{l})}(0,2^{l})^{r}\right]\\
    &\leq C_{2,r}M^{-r} 2^{-\delta mr} (3/2)^{(m-l)\delta r} \Lambda(2^l)^{r}\\
    &\leq C_{2,r}C_1^{r} M^{-r} 2^{-\delta mr} (3/2)^{(m-l)\delta r} 2^{\delta lr}
    \leq C_{2,r}C_1^{r}M^{-r} (3/4)^{(m-l)\delta r},
  \end{aligned}
\end{equation*}
where $C_1<\infty$ (depending only on $\beta$) and  $C_{2,r}<\infty$ (depending only on $\beta$ and $r$) are the constants defined in Lemmas \ref{Lam-order} and \ref{prop-rthmoment}, respectively.
Applying this to \eqref{tail-1} yields that for each $r\in\mathds{N}$,
\begin{equation}\label{tail-3}
  \mathds{P}\left[\max_{x\in[0,n)} R_{[0,n)}(0,x)>\widetilde{C}_1 M2^{\delta m}\right]\leq C_{2,r}C_1^{r}M^{-r}\sum_{l=0}^{m-1}2^{m-l}(3/4)^{(m-l)\delta r}.
\end{equation}
We now take $r>20/\delta$ such that $(3/4)^{\delta r}<(3/4)^{20}<1/4$. Applying this to \eqref{tail-3}, we can see that for any $n\in\mathds{N}$ and $M>0$,
\begin{equation}\label{tail-2}
  \mathds{P}\left[\max_{x\in[0,n)} R_{[0,n)}(0,x)>\widetilde{C}_1 M2^{\delta m}\right]\leq C_{2,r}C_1^{r}M^{-r}\sum_{l=0}^{m-1}2^{-(m-l)}\leq C_{2,r}C_1^{r}M^{-r}.
\end{equation}
Recall that $2^{m-1}\leq n<2^m$. Let $M'=\widetilde{C}_1 2^\delta c_1^{-1} M$, where $c_1>0$ is the constant (depending only on $\beta$) defined in Lemma \ref{Lam-order}. Then, by Lemma~\ref{Lam-order}, we have
$$
\widetilde{C}_1 M2^{\delta m}\leq M'c_{1}n^\delta\leq M'\Lambda(n).
$$
Combining this with \eqref{tail-2} and considering the arbitrariness of $M>0$, we obtain that for each $r>10/\delta$, there exists a constant $\widetilde{C}_{1,r}<\infty$ (depending on $r$ and $\beta$) such that for each $n\in\mathds{N}$ and  $M'>0$,
\begin{equation}\label{tailR-2r}
    \mathds{P}\left[\max_{x\in[0,n)} R_{[0,n)}(0,x)>M'\Lambda(n)\right]
    \leq C_{2,r}C_1^{r}(\widetilde{C}_1 2^\delta)^{r} c_1^{-r}(M')^{-r}=:\widetilde{C}_{1,r}(M')^{-r}.
\end{equation}
Consequently, for each $p\in \mathds{N}$, applying \eqref{tailR-2r} with $r>p\vee (20/\delta)$, we obtain that
\begin{equation*}
\begin{aligned}
  \Lambda(n)^{-p}\mathds{E}\left[\max_{x\in [0,n)}R_{[0,n)}(0,x)^{p}\right]
  &= \int_0^\infty \mathds{P}\left[\max_{x\in [0,n)}R_{[0,n)}(0,x)^{p}\geq t\Lambda(n)^{p}\right]\d t\\
  &= \int_0^\infty \mathds{P}\left[\max_{x\in [0,n)}R_{[0,n)}(0,x)\geq t^{\frac{1}{p}}\Lambda(n)\right]\d t\\
  &\leq \widetilde{C}_{1,r}+ \widetilde{C}_{1,r}\int_1^\infty t^{-\frac{r}{p}}\d t<\infty,
  \end{aligned}
\end{equation*}
which implies the desired result.
\end{proof}

\subsection{Proof of Theorem~\ref{thm-mr}}\label{sect-pThm1}

In this subsection, we primarily complete the proof of Theorem \ref{thm-mr} by verifying that the $\beta$-LRP model satisfies the conditions outlined in Propositions \ref{lem-td} and \ref{thm-quenched}.

Throughout this section, we will take the random graph $G$ from Section \ref{sect-pre} to be the $\beta$-LRP model, and we will use the resistance metric $R(\cdot,\cdot)$ as the metric $d(\cdot,\cdot)$ for $G$.
Additionally, we set $\varphi(r)=r$ and $\phi(r)=r^{1/\delta}$, where $\delta\in (0,1)$ is the exponent defined in Lemma \ref{Lam-order}. under this setting, for any $x\in \mathds{Z}$ and $r\geq 0$,  the ball centered at $x$ with radius $r$ with respect to the metric $d$ defined in \eqref{def-ball} is given by
\begin{equation}\label{def-RB}
B_r(x)=\left\{y\in \mathds{Z}:\ R(x,y)<r\right\}.
\end{equation}
Recall that $\mathds{V}_r(x)=\sum_{y\in B_r(x)}\text{deg}(y)$.

We start with estimating  the lower tail of $\mathds{V}_r(0)$.

\begin{lemma}\label{low-V}
  For all $\beta>0$, there exist constants $C_2<\infty$ and $ q_1>2$ {\rm(}both depending only on $\beta${\rm)} such that for any $r\geq 1$ and $\lambda\geq 1$,
  \begin{equation*}\label{eq-low-V}
    \mathds{P}\left[\mathds{V}_r(0)\leq \lambda^{-1} r^{1/\delta}\right]\leq C_2\lambda^{-q_1},
  \end{equation*}
  where $\delta\in (0,1)$ is the exponent defined in Lemma \ref{Lam-order}.
\end{lemma}

\begin{proof}
  For fixed $r\geq 1$ and $\lambda\geq 1$, without loss of generality, we assume that $1/\delta$ and $\lambda^{-1}r^{1/\delta}$ are integers. Otherwise, we can replace them with $\lfloor 1/\delta\rfloor$ and $\lfloor \lambda^{-1}r^{1/\delta}\rfloor$, respectively.

  Note that if $\mathds{V}_r(0)\leq \lambda^{-1}r^{1/\delta}$, then it follows that $\# B_r(0)\leq \lambda^{-1}r^{1/\delta}$. This implies that there exists $x\in (-\lambda^{-1}r^{1/\delta},\lambda^{-1}r^{1/\delta})$ such that $x\notin B_r(0)$ and $R(0,x)\geq r$. As a result, by the translation invariance of the LRP model, we obtain
  \begin{equation}\label{eq-low-1}
  \begin{aligned}
    \mathds{P}\left[\mathds{V}_r(0)\leq \lambda^{-1}r^{1/\delta}\right]&\leq \mathds{P}\left[\max_{x\in (-\lambda^{-1}r^{1/\delta},\lambda^{-1}r^{1/\delta})}R(0,x)\geq r\right]\\
    &\leq \mathds{P}\left[\max_{x\in (-\lambda^{-1}r^{1/\delta},0]}R(0,x)\geq r\right]+ \mathds{P}\left[\max_{x\in [0,\lambda^{-1}r^{1/\delta})}R(0,x)\geq r\right]\\
    &=2\mathds{P}\left[\max_{x\in [0,\lambda^{-1}r^{1/\delta})}R(0,x)\geq r\right].
    \end{aligned}
  \end{equation}
  In addition, using  Proposition~\ref{max-2r-moment} with $r=2/\delta$ and Markov's inequality, we have
  \begin{equation*}
    \mathds{P}\left[\max_{x\in [0,\lambda^{-1}r^{1/\delta})}R(0,x)\geq r\right]
    \leq C_{1,2/\delta}r^{-4/\delta}\Lambda(\lambda^{-1}r^{1/\delta})^{4/\delta}
    \leq \widetilde{C}_1\lambda^{-4},
  \end{equation*}
  where $C_{1,\cdot}$ is the constant defined in Proposition~\ref{max-2r-moment} and $\widetilde{C}_1$ is a constant depending only on $\beta$.
  Applying this result to \eqref{eq-low-1} yields the desired conclusion with $q_1=4$.
\end{proof}

From the lower tail of $\mathds{V}_r(0)$, we can immediately derive the following corollary, which corresponds to condition \eqref{exp-V} in Remark~\ref{remark}.

\begin{corollary}\label{low-V-2}
  For all $\beta>0$, there exists a constant $C_3<\infty$ {\rm(}depending only on $\beta${\rm)} such that for each $r\geq 1$ and $\lambda\geq 1$,
  \begin{equation*}\label{eq-low-V-2}
    \mathds{E}\left[1/\mathds{V}_r(0)\right]\leq C_3r^{-1/\delta}.
  \end{equation*}
\end{corollary}

Next, we consider the upper tail of $\mathds{V}_r(0)$.

\begin{lemma}\label{up-V}
For all $\beta>0$,  there exist constants $c_2,\ q_2>0$ {\rm(}both depending only on $\beta${\rm)} such that for any $r\geq 1$ and $\lambda\geq 1$,
  \begin{equation*}\label{eq-up-V}
    \mathds{P}\left[\mathds{V}_r(0)\geq \lambda r^{1/\delta}\right]\leq c_2\lambda^{-q_2}.
  \end{equation*}
\end{lemma}

\begin{proof}
  Note that for any $r\geq 1$ and $n\in \mathds{N}$, if $R(0,[-n,n]^c)\geq r$, then it follows from the monotonicity of the effective resistance that
  $$
  R(0,x)\geq r\quad \text{for each } x\in [-n,n]^c,
$$
   which implies that $B_r(0)\subset [-n,n]$.
  Consequently, we conclude that if $B_r(0)\nsubseteq [-n,n]$, then $R(0,[-n,n]^c)<r$. This leads us to
  \begin{equation}\label{up-V-1}
    \begin{aligned}
      \mathds{P}\left[\mathds{V}_r(0)\geq \lambda r^{1/\delta}\right]
      &\leq \mathds{P}[B_r(0)\nsubseteq [-n,n]]+\mathds{P}\left[B_r(0)\subset [-n,n],\ \mathds{V}_r(0)\geq \lambda r^{1/\delta}\right]\\
      &\leq \mathds{P}\left[R(0,[-n,n]^c)<r\right]+\mathds{P}\left[\sum_{x\in[-n,n]}\mathrm{deg}(x)\geq \lambda r^{1/\delta}\right]
    \end{aligned}
  \end{equation}
for each $n\in \mathds{N}$.

  We now set $n=\lfloor \frac{1}{8\e\mu_\beta}\lambda r^{1/\delta}\rfloor $, where $\mu_\beta:=\mathds{E}[\mathrm{deg}(0)]$.
  Let $\mathcal{E}_n$ be the set of all edges that have at least one endpoint in $[-n,n]$. Then, we have
  $$
  \#\mathcal{E}_n\leq \sum_{x\in[-n,n]}\mathrm{deg}(x)\leq 2\#\mathcal{E}_n,
  $$
  Combining this with the translation invariance of the LRP model yields
  $$
  \widetilde{\mu}_n:=\mathds{E}[\#\mathcal{E}_n]\leq \sum_{x\in[-n,n]}\mathds{E}[\mathrm{deg}(x)]\leq (2n+1)\mu_\beta.
   $$
   Consequently, we obtain that
  \begin{equation}\label{up-V-2}
    \begin{aligned}
      \mathds{P}\left[\sum_{x\in[-n,n]}\mathrm{deg}(x)\geq \lambda r^{1/\delta}\right]
      &\leq \mathds{P}\left[\#\mathcal{E}_n\geq \lambda r^{1/\delta}/2\right]
      \leq \left(\frac{\e\widetilde{\mu}_n}{\lambda r^{1/\delta}/2}\right)^{\lambda r^{1/\delta}/2}\\
      &\leq (2\e(2n+1)\mu_\beta\lambda^{-1} r^{-1/\delta})^{\lambda r^{1/\delta}/2}\leq 2^{-\lambda r^{1/\delta}/2}\leq 2^{-\lambda/2},
    \end{aligned}
  \end{equation}
   where the second inequality  follows from Chernoff's inequality, since $\#\mathcal{E}_n$ is the sum of a series of independent Bernoulli variables.
   Additionally, from \cite[Corollary~1.2]{DFH25+} and the fact that $n=\lfloor \frac{1}{8\e\mu_\beta}\lambda r^{1/\delta}\rfloor $, there exist constants $c_0, q_0>0$ (both depending only on $\beta$) such that
  \begin{equation}\label{up-V-3}
    \mathds{P}\left[R(0,[-n,n]^c)\geq r\right]\leq \left(rn^{-\delta}\right)^{q_0}\leq c_0\lambda^{-q_0\delta},
  \end{equation}
  where $\delta\in (0,1)$ is the exponent defined in Lemma \ref{Lam-order}. By plugging \eqref{up-V-2} and \eqref{up-V-3} into \eqref{up-V-1}, we complete the proof.
\end{proof}

We now  turn to estimate the lower tail of $R(0,B_r(0)^c)$.

\begin{lemma}\label{low-R}
For all $\beta>0$, there exist constants $c_3, q_3>0$ {\rm(}both depending only on $\beta${\rm)} such that for any $r\geq 1$ and $\lambda\geq 1$,
  \begin{equation*}\label{eq-low-R}
    \mathds{P}\left[R(0,B_r(0)^c)\leq \lambda^{-1} r\right]\leq c_3\lambda^{-q_3}.
  \end{equation*}
\end{lemma}
\begin{proof}
For fixed $r\geq 1$ and $\lambda\geq 1$, denote $M=\lambda^{-1/(2\delta)}r^{1/\delta}$.
Note that if $[-M,M]\nsubseteq B_r(0)$, then by the definition of $B_r(0)$ in \eqref{def-RB} we have  $\max_{x\in[-M,M]}R(0,x)\geq r$. Thus we obtain that
  \begin{equation}\label{eq-low-r-1}
    \begin{aligned}
      &\mathds{P}\left[R(0,B_r(0)^c)\leq\lambda^{-1} r\right]\\
      &\leq \mathds{P}\big[[-M,M]\nsubseteq B_r(0)\big]
      +\mathds{P}\left[R(0,B_r(0)^c)\leq \lambda^{-1} r,\ [-M,M]\subseteq B_r(0)\right]\\
      &\leq \mathds{P}\left[\max_{x\in[-M,M]}R(0,x)\geq r\right]+\mathds{P}\left[R(0,B_r(0)^c)\leq \lambda^{-1} r,\  [-M,M]\subseteq B_r(0)\right]\\
      &\leq \mathds{P}\left[\max_{x\in[-M,M]}R(0,x)\geq r\right]+\mathds{P}\left[R(0,[-M,M]^c)\leq \lambda^{-1} r\right].
    \end{aligned}
  \end{equation}

  Next, we will estimate each term on the right-hand side of \eqref{eq-low-r-1}.
  For the first term,
  by Proposition~\ref{max-2r-moment}, there exist constants $\widetilde{C}_1, \widetilde{C}_2<\infty$ (both depending only on $\beta$) such that
  \begin{equation*}
  \begin{aligned}
    \mathds{E}\left[\max_{x\in[-M,M]}R(0,x)^2\right]&\leq \mathds{E}\left[\max_{x\in[0,M]}R_{[0,M]}(0,x)^2\right]+\mathds{E}\left[\max_{x\in[-M,0]}R_{[-M,0]}(0,x)^2\right]\\
    &\leq \widetilde{C}_1\Lambda(M)^2\leq \widetilde{C}_2 M^{2\delta}.
    \end{aligned}
  \end{equation*}
Combining this with Markov's inequality, we get that
  \begin{equation}\label{eq-low-r-2}
    \mathds{P}\left[\max_{x\in[-M,M]}R(0,x)\geq r\right]\leq \widetilde{C}_2r^{-2}M^{2\delta}.
  \end{equation}

For the second term on the right-hand side of \eqref{eq-low-r-1}, from \cite[Corollary~1.2]{DFH25+}, we can see that there exists a constant $q_0>0$ (depending only on $\beta$) such that
  \begin{equation}\label{eq-low-r-3}
    \mathds{P}\left[R(0,[-M,M]^c)\leq \lambda^{-1} r\right]\leq \left(\lambda^{-1}r/M^\delta\right)^{q_0}.
  \end{equation}
  Applying \eqref{eq-low-r-2} and \eqref{eq-low-r-3} to \eqref{eq-low-r-1}, along with the choice that $M= \lambda^{-1/(2\delta)}r^{1/\delta}$, we obtain that
  \begin{equation*}
    \mathds{P}\left[R(0,B_r(0)^c)\leq\lambda^{-1} r\right]\leq \widetilde{C}_2r^{-2}M^{2\delta}+\left(\lambda^{-1}r/M^\delta\right)^{q_0}\leq \widetilde{C}_2\lambda^{-1}+\lambda^{-q_0/2}.
  \end{equation*}
  Hence, we conclude with the desired result, setting $c_3=\widetilde{C}_2+1$ and $q_3=\min\{q_0/2,1\}$.
\end{proof}

With the above lemmas at hand, we can check some assumptions in  {\bf Assumption A} holds for the $\beta$-LRP model as follows.
\begin{proposition}\label{A1A3}
  For all $\beta>0$, assumptions {\rm(A1)} and {\rm(A3)} in {\bf Assumption A} hold for the $\beta$-LRP model with $\varphi(r)=r$ and $\phi(r)=r^{1/\delta}$,  where $\delta\in (0,1)$ is the exponent defined in Lemma \ref{Lam-order}.
\end{proposition}
\begin{proof}
  For $\lambda>1$, recall the definition of $J(\lambda)$ in Definition \ref{def-Jlambda} and we take $d(\cdot,\cdot)=R(\cdot,\cdot)$ in the LRP model. Combining this with the fact that $\varphi(r)=r$, we find that the following condition in $J(\lambda)$ is trivial:
  \begin{equation*}
  R(0,y)\leq \lambda \varphi(d(0,y))\quad \text{for all $r\geq 1$ and all } y\in B_r(0).
  \end{equation*}
 As a result, we have
  \begin{equation*}
    \mathds{P}\left[r\in J(\lambda)\right]\geq 1-\left(\mathds{P}\left[\mathds{V}_r(0)>\lambda r^{1/\delta}\right]+\mathds{P}\left[\mathds{V}_r(0)<\lambda^{-1}r^{1/\delta}\right]+\mathds{P}\left[R(0,B_r(0)^c)<\lambda^{-1}r\right]\right).
  \end{equation*}
  Combining this with Lemmas~\ref{low-V}, \ref{up-V} and \ref{low-R} yields the desired result.
\end{proof}

With the help of Propositions~\ref{lem-td}, \ref{thm-quenched} and Remark~\ref{remark}, we can now present the proof of  Theorem~\ref{thm-mr}.

\begin{proof}[Proof of Theorem~\ref{thm-mr}]
  Recall that $\phi(r)=r^{1/\delta}$ and $\varphi(r)=r$. Therefore,
   $$
   \Psi(r):=(\phi\cdot\varphi)^{-1}(r)=r^{\delta/(1+\delta)}\quad \text{and}\quad \phi(\Psi(r))=r^{1/(1+\delta)}.
    $$
    From Proposition~\ref{A1A3}, we see that (A1) and (A3) in {\bf Assumption A} hold and \eqref{cond-phi2} holds for $\alpha=1, D=\delta^{-1}, C_3=C_4=1$ and any $m_1,m_2>0$. Consequently, from Proposition~\ref{thm-quenched}, we conclude that Theorem \ref{thm-mr} (1) holds and for each $x\in \mathds{Z}$, $\mathds{P}$-a.s.
  $$
  d^{(q)}_s(\beta)=\frac{2D}{D+\alpha}=\frac{2}{1+\delta}.
  $$

  Next, we turn to the annealed spectral dimension.
  As mentioned in Remark~\ref{remark} (2), we have that $J(\lambda)=\widehat{J}(\lambda)$ in our setting. Thus, it follows from Proposition~\ref{A1A3} that (A1) also holds with respect to $\widehat{J}(\lambda)$.
  Combining this with Corollary~\ref{low-V-2} and using Remark~\ref{remark} (1), we obtain that both \eqref{an-lowb} and \eqref{an-upb} in Proposition \ref{lem-td} hold. This implies that there exist constants $0<\widetilde{c}_1<\widetilde{c}_2<\infty$ (depending only on $\beta$) such that for any $n\in \mathds{N}$,
  \begin{equation*}
    \frac{\widetilde{c}_1}{n^{1/(1+\delta)}}=\frac{\widetilde{c}_1}{\phi(\Psi(n))}\leq \mathds{E}\left[p_{2n}(0,0)\right]\leq \frac{\widetilde{c}_2}{\phi(\Psi(n))}=\frac{\widetilde{c}_2}{n^{1/(1+\delta)}}.
  \end{equation*}
  As a result, the annealed spectral dimension of the $\beta$-LRP model is given by
  $$d_s^{(a)}(\beta)=-2\lim_{n\to\infty}\frac{\log\mathds{E}[p_{2n} (0,0)]}{\log n}=\frac{2}{1+\delta}.$$
  Thus, we complete the proof.
\end{proof}

\bigskip

\noindent{\bf Acknowledgement.} \rm
We warmly thank Jian Ding and Takashi Kumagai for helpful discussions.
L.-J.\ Huang is partially supported by National Key R\&D Program of China No.\ 2023YFA1010400, and the National Natural Science Foundation of China No.\ 12471136.

\bibliographystyle{plain}
\bibliography{resistance-ref2}

\end{document}